\documentclass[11pt]{amsart}

\usepackage{geometry}
\usepackage[T1]{fontenc}
\usepackage{amsmath,amssymb}
\usepackage{tikz-cd}
\usepackage[protrusion=true,expansion=false]{microtype}
\usepackage[hidelinks]{hyperref}

\allowdisplaybreaks[2]
\newtheorem{theorem}{Theorem}[section]
\newtheorem{proposition}[theorem]{Proposition}
\newtheorem{lemma}[theorem]{Lemma}
\newtheorem{corollary}[theorem]{Corollary}
\theoremstyle{definition}
\newtheorem{definition}[theorem]{Definition}
\theoremstyle{remark}
\newtheorem{remark}[theorem]{Remark}
\newtheorem{question}[theorem]{Question}

\newtheorem*{fujitaconjecture}{Fujita's freeness conjecture}

\DeclareMathOperator{\lct}{lct}
\DeclareMathOperator{\mult}{mult}
\DeclareMathOperator{\edim}{edim}
\DeclareMathOperator{\ord}{ord}
\newcommand{\m}{\mathfrak m}
\newcommand{\ebarzero}{\overline{e}_0}
\newcommand{\ebarone}{\overline{e}_1}

\title{A linear bound for Fujita's freeness conjecture}
\author{Jingjun Han}
\address{Shanghai Center for Mathematical Sciences \& School of Mathematical Sciences,
Fudan University, Shanghai 200438, China}
\email{hanjingjun@fudan.edu.cn}
\date{}

\begin{document}

\begin{abstract}
Let $X$ be a smooth complex projective variety of dimension $n$, and let $L$ be an ample Cartier divisor. We prove that $K_X+mL$ is globally generated for every integer $m\geq\lceil C_0n\rceil$, where $C_0=1.77629\ldots$ is an explicit constant. In particular, $K_X+2nL$ is globally generated. We also prove that if $Z$ is a normal projective variety of dimension $n$ with $K_Z$ Cartier, nef and big, then $|mK_Z|$ defines a birational map for every integer $m\geq1+\lceil C_0n+3\sqrt{2\mathrm{e}n}\rceil$. Our main input is a new estimate for the multiplicity of a minimal log canonical center. If $(X,\Delta)$ is log canonical near a closed point $x$ but is not klt at $x$, and $W$ is the positive-dimensional minimal log canonical center through $x$, then $2\overline{e}_1(\mathfrak m_{W,x})\leq\bigl(\dim W-\operatorname{lct}_x((X,\Delta);\mathfrak m_x)\bigr)\operatorname{mult}_xW$, where $\overline{e}_1(\mathfrak m_{W,x})$ is the first normal Hilbert coefficient of the maximal ideal of $\mathcal O_{W,x}$. This implies $\operatorname{mult}_xW\leq \frac{(a+c)^{a+c}}{a^ac^c}$, where $a:=\frac{\dim W-\operatorname{lct}_x((X,\Delta);\mathfrak m_{x})}{2}$ and $c:=\operatorname{edim}\mathcal O_{W,x}-\dim W$.
\end{abstract}

\maketitle

\section{Introduction}

Throughout, we work over the field of complex numbers $\mathbb C$.


Globally generated line bundles give rise to morphisms to projective space, which are among the basic tools for studying projective varieties. Fujita's freeness conjecture predicts a sharp uniform bound for the global generation of adjoint line bundles and has therefore received considerable attention. 

\begin{fujitaconjecture}[{\cite{Fuj87}}]
Let $X$ be a smooth projective variety of dimension $n$, and let $L$ be an ample Cartier divisor. Then $K_X+mL$ is globally generated for every integer $m\geq n+1$.
\end{fujitaconjecture}


Fujita's conjecture is known in dimensions at most five \cite{Rei88,EL93,Kaw97,YZ17,YZ20}. Angehrn--Siu proved the global generation for $m\geq(n^2+n+2)/2$ \cite{AS95}. Another induction leading to a quadratic bound was developed by Helmke  \cite{Hel97}. Heier later combined the Angehrn--Siu and Helmke
approaches to obtain a bound of order $n^{4/3}$ \cite{Hei02}. Recently, Ghidelli and Lacini proved that global generation holds for every integer $m\geq n(\log\log n+2.34)$ when $n\geq2$ \cite[Theorem~1.1]{GL24}. Thus the previously known general bounds in arbitrary dimension were still superlinear.

We prove the following linear bound in arbitrary dimension. 

\begin{theorem}\label{thm:main}
Let $X$ be a smooth projective variety of dimension $n$ and let $L$ be an ample Cartier divisor. Then $K_X+mL$ is globally generated for every integer $m\geq\lceil C_0n\rceil$, where $\phi:=(1+\sqrt5)/2$, and
\[
 C_0:=\int_0^{\phi^{-2}}
 \frac{2\bigl(-z\log z-(1-z)\log(1-z)\bigr)}
 {z^{3/2}\log^2 z}\,dz
 =1.77629988\ldots .
\]
\end{theorem}


\begin{corollary}\label{cor:2n}
Let $X$ be a smooth projective variety of dimension $n$ and let $L$ be an ample Cartier divisor. Then $K_X+2nL$ is globally generated.
\end{corollary}

To the best of our knowledge, Theorem~\ref{thm:main} gives the first
uniform bound for Fujita's freeness conjecture that is linear in the
dimension. Thus the remaining gap from the conjectural bound $n+1$
lies in the leading constant rather than in the order of growth. Corollary~\ref{cor:2n} gives the simpler bound $2n$.


\medskip

Quadratic birationality bounds for adjoint linear systems were
obtained in \cite{EKL95,DC14}. We also prove the following linear bound for pluricanonical systems when the canonical divisor is Cartier, nef and big. 

\begin{theorem}\label{thm:pluricanonical-birationality}
Let $Z$ be a normal projective variety of dimension $n$. Assume that $K_Z$ is Cartier, nef and big. Then $|mK_Z|$ defines a birational map for every integer $m\geq1+\lceil C_0n+3\sqrt{2\mathrm{e}n}\rceil$.
\end{theorem}

In particular, Theorem~\ref{thm:pluricanonical-birationality}
applies to Gorenstein minimal models of general type.
For explicit pluricanonical birationality bounds in dimension
three, see \cite{Che01,CCZ07,Che18}. For related uniform birationality results without explicit bounds, see \cite{HM06,Tak06,Tsu07,Pac09,BZ16,Bir23}. 

We also prove bounds for pluricanonical and anti-pluricanonical
systems in terms of the Cartier index, and more generally for
polarized pairs, see Corollaries~\ref{cor:pluricanonical-systems}
and~\ref{cor:polarized-birationality}. For related results on
polarized pairs and Iitaka fibrations, see \cite{Pac09,BZ16,Bir23}. For point separation by adjoint linear systems and the corresponding
homeomorphism results, see Theorem~\ref{thm:point-separation}
and Corollary~\ref{cor:adjoint-morphisms}, respectively.





\medskip

We outline the proof of Theorem~\ref{thm:main}.
Fix a closed point $x\in X$.
It suffices to construct an effective $\mathbb Q$-divisor
$\Delta\sim_{\mathbb Q}sL$, with $s<C_0n$, such that
$(X,\Delta)$ is log canonical near $x$ and $\{x\}$ is
the unique lc center through $x$.
Nadel vanishing then produces a global section of $K_X+mL$
which does not vanish at $x$ for every integer
$m\geq\lceil C_0n\rceil$.
We construct such a divisor by successively cutting down
minimal lc centers through $x$ to $\{x\}$ and then
applying tie-breaking.

Suppose that $\Delta\sim_{\mathbb Q}sL$ is effective,
$(X,\Delta)$ is log canonical near $x$, and its minimal
lc center $W$ through $x$ has positive dimension.
Set $d:=\dim W$, and let
$b:=b_x(X,\Delta)=\lct_x((X,\Delta);\m_x)$ be the local
discrepancy.
The multiplicity $\mult_xW$ enters through the comparison
of asymptotic Riemann--Roch on $W$ with the Hilbert--Samuel
polynomial at $x$.
More precisely, if a positive rational number $q$ satisfies
$$
q^d(L^d\cdot W)>b^d\mult_xW,
$$
then there exist an effective $\mathbb Q$-divisor
$D\sim_{\mathbb Q}L$ and a rational number $0<t<q$ such
that, for $\Delta':=\Delta+tD$, the pair $(X,\Delta')$
is log canonical near $x$ and its minimal lc center $W'$
through $x$ satisfies $W'\subsetneq W$
(cf. \cite[Proposition~3.2]{GL24}).
Since $\Delta'\sim_{\mathbb Q}(s+t)L$, replacing $\Delta$
by $\Delta'$ decreases the dimension of the minimal lc
center through $x$ and increases $s$ by $t<q$.

Although $X$ is smooth, $W$ need not be smooth at $x$,
so $\mult_xW$ need not equal $1$.
In Koll\'ar's algebraic formulation of the proof of
Angehrn--Siu \cite{AS95,Kol97}, one first chooses an
effective $\mathbb Q$-divisor
$\Delta_y\sim_{\mathbb Q}qL|_W$ with $\ord_y\Delta_y>d$ at a
general smooth point $y\in W$.
Since $\mult_yW=1$, such a divisor exists whenever
$q^d(L^d\cdot W)>d^d$.
One then takes a limit of these divisors as $y$ approaches
$x$, lifts the limiting divisor to $X$, and perturbs
the boundary to obtain a pair whose minimal lc center
through $x$ is strictly contained in $W$, see
\cite[Theorem~6.8.1 and Corollary~7.8]{Kol97}.
This avoids the need to estimate $\mult_xW$ and provides a quadratic bound for Fujita's conjecture. Helmke instead works at the fixed point $x$ and proves
the multiplicity bound \cite[Theorem~4.3]{Hel97}
\begin{equation}\label{eq:previous-multiplicity}
\mult_xW\leq
\binom{\edim\mathcal O_{W,x}-\lceil b\rceil}
{\edim\mathcal O_{W,x}-d}.
\end{equation}
The main new ingredient in the proof of
Theorem~\ref{thm:main} is the following multiplicity estimate.

\begin{theorem}\label{thm:multiplicity-bound} Let $(X,\Delta)$ be a pair that is log canonical near
a closed point $x$ and is not klt at $x$. Let $W$ be the positive-dimensional minimal lc center
of $(X,\Delta)$ through $x$.
Then $0<b\leq d$, and
$$
\mult_xW\leq\frac{(a+c)^{a+c}}{a^ac^c},
$$
where $d:=\dim W$, $b:=b_x(X,\Delta)$, $a:=(d-b)/2$,
and $c:=\edim\mathcal O_{W,x}-d$.
\end{theorem}

Here and below, the quotient
$\frac{(a+c)^{a+c}}{a^ac^c}$ is defined to be $1$
if $a=0$ or $c=0$.


We also obtain the following bound for the multiplicity
of a klt singularity.

\begin{corollary}\label{cor:klt-multiplicity}
Let $(X,\Delta)$ be a klt pair near a closed
point $x$. Set $d:=\dim X$, $b:=b_x(X,\Delta)$, $a:=(d-b)/2$, and
$c:=\edim\mathcal O_{X,x}-d$.
Then $0<b\leq d$. If $c=0$, then $\mult_xX=1$. If $c>0$, then
$$\mult_xX\leq \left\lfloor \frac{1}{s+1-a}\binom{c+s+1}{c+1}
\right\rfloor,$$
where $s:=\left\lfloor\frac{(c+1)a}{c}\right\rfloor$.
\end{corollary}

For a klt germ $x\in X$, an estimate of the same form as \eqref{eq:previous-multiplicity} holds with $W$ replaced
by $X$, $d=\dim X$, and $b=b_x(X,0)$
(cf. \cite[Propositions~2.6 and~3.2]{TW04}, \cite[Proposition~3.9]{Shi17}).
Corollary~\ref{cor:klt-multiplicity} refines this estimate. For minimal lc centers in smooth varieties, Corollary~\ref{cor:discrete-lc-center} refines both
\eqref{eq:previous-multiplicity} and the bound in Theorem~\ref{thm:multiplicity-bound}. The proofs and the comparison of the bounds are given
in Appendix~\ref{sec:discrete-refinements}.



\medskip

We now return to the proof of Theorem~\ref{thm:main}
and apply Theorem~\ref{thm:multiplicity-bound} to bound
the sum of the coefficients added in the construction
of $\Delta$. Set $\Delta_0:=0$, $W_0:=X$, and $b_0=d_0:=n$.
For $i>0$, let $W_i$ be the minimal lc center of
$(X,\Delta_i)$ through $x$, and set
$d_i:=\dim W_i$ and $b_i:=b_x(X,\Delta_i)$.
As long as $d_i>0$, for every $\varepsilon_i>0$ we may
choose an effective $\mathbb Q$-divisor
$D_{i+1}\sim_{\mathbb Q}L$ such that, setting
$t_{i+1}:=\lct_x((X,\Delta_i);D_{i+1})$ and
$\Delta_{i+1}:=\Delta_i+t_{i+1}D_{i+1}$,
the pair $(X,\Delta_{i+1})$ is log canonical near $x$,
one has $W_{i+1}\subsetneq W_i$, and
\begin{equation}\label{eqn: upperbound ti}
t_{i+1}\leq
(b_i-b_{i+1}+\varepsilon_i)
\left(
\frac{\mult_xW_i}{L^{d_i}\cdot W_i}
\right)^{1/d_i}
\end{equation}
(cf. \cite[Proposition~3.2 and Remark~3.3]{GL24}).
Since the dimensions decrease strictly, after at most
$n$ steps we have $W_k=\{x\}$.
Moreover, $n=b_0\geq b_1\geq\cdots\geq b_k=0$.

For $0\leq i<k$, set $a_i:=(d_i-b_i)/2$ and
$c_i:=\edim\mathcal O_{W_i,x}-d_i$.
Then $0\leq c_i\leq n-d_i$, and $a_0=c_0=0$.
Applying Theorem~\ref{thm:multiplicity-bound}, we obtain
$$
\sum_{i=0}^{k-1}t_{i+1}
\leq
\sum_{i=0}^{k-1}(b_i-b_{i+1}+\varepsilon_i)
\left(
\frac{(a_i+c_i)^{a_i+c_i}}
{a_i^{a_i}c_i^{c_i}}
\right)^{1/d_i}
<C_0n.
$$
The last inequality follows from
Proposition~\ref{prop:numerical-sum} for sufficiently small
$\varepsilon_i>0$.
For comparison, using Helmke's multiplicity bound instead
gives $\sum_{i=0}^{k-1}t_{i+1}<n(\log\log n+2.34)$ \cite[Theorem~4.1\textup{(1)}]{GL24}.

Since $D_{i+1}\sim_{\mathbb Q}L$, we have
$\Delta_k\sim_{\mathbb Q}
\bigl(\sum_{i=0}^{k-1}t_{i+1}\bigr)L$.
By tie-breaking (cf. \cite[Lemma~2.8]{GL24}), we may perturb
$\Delta_k$ to obtain the required effective
$\mathbb Q$-divisor $\Delta\sim_{\mathbb Q}sL$
with $s<C_0n$.
This proves Theorem~\ref{thm:main}.

\medskip

We next explain the main ingredients in the proof of
Theorem~\ref{thm:multiplicity-bound}. For simplicity, we may assume that $X$ is projective, $W$ is normal and $d:=\dim W\geq2$. Since $W$ is an lc center of the log canonical pair $(X,\Delta)$, it is natural to apply the (higher-codimensional) subadjunction formula from birational geometry:
\begin{equation*}
(K_X+\Delta)|_W\sim_{\mathbb Q}K_W+B_W+M_W,
\end{equation*}
where $B_W$ and $M_W$ are the traces of the discriminant and moduli $b$-divisors $\mathbf B$ and $\mathbf M$, respectively. The coefficients of $\mathbf B$ are defined by log canonical thresholds, while there is a suitable birational model $Z'\to W$ such that $M_{Z'}$ is nef.

Let $f\colon Y\to W$ be the normalized blow-up of $\m_{W,x}$, write $\m_{W,x}\mathcal O_Y=\mathcal O_Y(-E)$ and $H:=-E$, and let $B_Y$ be the trace of $\mathbf B$ on $Y$. Proposition~\ref{prop:threshold} compares the local discrepancy with the coefficients of $\mathbf B$ and shows that
\begin{equation*}
 1-\mult_FB_Y
 \geq
 b_x(X,\Delta)\ord_F(\m_{W,x})
\end{equation*}
for every prime component $F$ of $E$. The nefness of $M_{Z'}$ yields $(K_Y+B_Y)\cdot H^{d-1}\geq0$. On the other hand, the comparison of Riemann--Roch on $Y$ with the normal Hilbert polynomial of $\m_{W,x}$ in Proposition~\ref{prop:projective-rr} yields
\begin{equation*}
 K_Y\cdot H^{d-1}
 =(d-1)\mult_xW-2\ebarone(\m_{W,x}).
\end{equation*}
Combining these estimates, we obtain the following inequality relating the first normal Hilbert coefficient to
the local discrepancy $b_x(X,\Delta)$.

\begin{theorem}\label{thm:local}
Let $(X,\Delta)$ be a pair that is log canonical near
a closed point $x$ and is not klt at $x$.
Let $W$ be the positive-dimensional minimal lc center
of $(X,\Delta)$ through $x$. Then 
$$
2\ebarone(\m_{W,x})\leq\bigl(\dim W-b_x(X,\Delta)\bigr)\mult_xW,
$$
where $\ebarone(\m_{W,x})$ is the first normal Hilbert coefficient of the maximal ideal of $\mathcal O_{W,x}$. 
\end{theorem}

Since $W$ is a minimal lc center, $\mathcal O_{W,x}$ is Cohen--Macaulay. Combining Theorem~\ref{thm:local}, the inequality $e_1(\m_{W,x})\leq\ebarone(\m_{W,x})$, and Corollary~\ref{cor:local-algebra} proves Theorem~\ref{thm:multiplicity-bound}.

\medskip
\noindent\emph{Acknowledgments.} The author would like to thank
Guodu Chen, Meng Chen, Chen Jiang, Yuchen Liu, Yujie Luo, Fanjun Meng, Chenyang Xu, and Zhixian Zhu for fruitful discussions. This work was supported by the National Key R\&D Program of China (No.~2025YFA1018100, No.~2023YFA1010600) and the NSFC for Excellent Young Scientists (No.~12322102). The author is a member of LMNS, Fudan University. The author thanks ChatGPT 5.6 Pro for suggesting the examples in Appendix ~\ref{sec:discrete-refinements} and for assistance with English editing, and Maple for computing $C_0$.


\section{Preliminaries and adjunction}

We follow the standard notation and conventions of \cite{KM98} for singularities of pairs. By a birational model of a normal variety $V$, we mean a normal variety $V'$ equipped with a proper birational morphism $V'\to V$. Unless otherwise stated, all varieties are quasi-projective over
$\mathbb C$, and all pairs and sub-pairs are $\mathbb Q$-pairs and $\mathbb Q$-sub-pairs.

\subsection{Pairs, singularities, and the local discrepancy}
\begin{definition}
A \emph{pair} $(X,\Delta)$ consists of a normal variety $X$ and an effective $\mathbb Q$-divisor $\Delta$ such that $K_X+\Delta$ is $\mathbb Q$-Cartier. If $\Delta$ is not assumed effective, we call $(X,\Delta)$ a $\mathbb Q$-sub-pair, or simply a \emph{sub-pair}. A germ $x\in(X,\Delta)$ consists of a pair and a closed point $x\in X$. For a divisor $D$ on a normal variety and a prime divisor $F$, we write $\mult_FD$ for the coefficient of $F$ in $D$. If $F$ is a prime divisor on a birational model $\pi\colon X'\to X$, write
\[
 K_{X'}+\Delta_{X'}=\pi^*(K_X+\Delta),
 \qquad
 A_{X,\Delta}(F):=1-\mult_F\Delta_{X'}.
\]


The sub-pair is \emph{sub-lc} if $A_{X,\Delta}(F)\geq0$ for every prime divisor $F$ over $X$. A pair is \emph{log canonical} (respectively, \emph{klt}) if $A_{X,\Delta}(F)\geq0$ (respectively, $>0$) for every prime divisor $F$ over $X$. A prime divisor $F$ with $A_{X,\Delta}(F)=0$ is an \emph{lc place}, and its center $c_X(F)$ on $X$ is an \emph{lc center}. A prime divisor $F$ with $A_{X,\Delta}(F)\leq0$ is a \emph{non-klt place}, and its center on $X$ is a \emph{non-klt center}. 

For an effective $\mathbb Q$-divisor $\Delta$ on a smooth variety $X$, we denote by $\mathcal J(X,\Delta)$ its multiplier ideal \cite[\S9.2]{Laz04}. 
\end{definition}

\begin{definition}
Assume that $(X,\Delta)$ is log canonical near a point $x\in X$. For a nonzero coherent ideal $\mathfrak a\subset\mathcal O_X$, its \emph{local log canonical threshold at $x$} is
\[
 \lct_x\bigl((X,\Delta);\mathfrak a\bigr)
 :=\inf_{\substack{F\text{ prime over }X\\
 x\in c_X(F),\ \ord_F(\mathfrak a)>0}}
 \frac{A_{X,\Delta}(F)}{\ord_F(\mathfrak a)}.
\]
For an effective $\mathbb Q$-Cartier $\mathbb Q$-divisor $D$, we similarly write
\[
 \lct_x\bigl((X,\Delta);D\bigr)
 :=\sup\{t\in\mathbb Q_{\geq0}\mid (X,\Delta+tD)
 \text{ is log canonical near }x\}.
\]
\end{definition}

When $X$ is smooth, the following definition is due to Helmke \cite{Hel97} and Ein \cite{Ein97}.
\begin{definition}
Assume that $(X,\Delta)$ is log canonical near a closed point $x$. Let $\mathfrak m_x\subset\mathcal O_X$ be the ideal sheaf of $x$. We call
\[
 b_x(X,\Delta)
 :=\lct_x\bigl((X,\Delta);\mathfrak m_x\bigr)
\]
the \emph{local discrepancy} of $(X,\Delta)$ at $x$.
\end{definition}

\subsection{Minimal lc centers and normalization}

\begin{definition}
Assume that $(X,\Delta)$ is log canonical near $x$ and is not klt at $x$. By \cite[Proposition~4.8\textup{(i)}, \textup{(ii)}]{Amb03} (cf. \cite[Theorem~9.1\textup{(2)}, \textup{(4)}]{Fuj11}), among the lc centers through $x$ there is a unique minimal one. We denote it by $W_x(X,\Delta)$ and call it the \emph{minimal lc center of $(X,\Delta)$ at $x$}. When the pair and the point $x$ are fixed, we abbreviate $W_x(X,\Delta)$ to $W$.
\end{definition}

The following consequence of \cite[Theorem~7.2]{FG12} will be used below.

\begin{lemma}\label{lem:minimal-lc-center}
Let $(X,\Delta)$ be a pair that is log canonical near $x\in X$ and is not klt at $x$. Then $W_x(X,\Delta)$ is normal with rational singularities near $x$. In particular, $\mathcal O_{W_x(X,\Delta),x}$ is Cohen--Macaulay.
\end{lemma}

\begin{proof}
After shrinking $X$ around $x$, we may assume that $X$ is affine, $(X,\Delta)$ is log canonical, and $W:=W_x(X,\Delta)$ is a minimal lc center of $(X,\Delta)$. The assertion follows from \cite[Theorem~7.2]{FG12}.
\end{proof}

For a pair $(X,\Delta)$ and a point $x$ as above, set $W:=W_x(X,\Delta)$, and let $\nu\colon Z\to W$ be its normalization. Lemma~\ref{lem:minimal-lc-center} implies that $\nu$ is an isomorphism near $x$, so $\mathcal O_{Z,x}\simeq\mathcal O_{W,x}$. 


\subsection{Discriminant and moduli \texorpdfstring{$b$}{b}-divisors}

\begin{definition}
Let $Z$ be a normal variety. We use $b$-divisors as in \cite[Section~2]{FH23}. If $\mathbf D$ is a $\mathbb Q$-Weil $b$-divisor over $Z$ and $Z'\to Z$ is a birational model, we denote its trace on $Z'$ by $D_{Z'}$. If $F$ is a prime divisor over $Z$, we set $\mult_F\mathbf D:=\mult_F D_{Z'}$, where $Z'$ is any model on which $F$ appears. If $D_{Z'}$ is $\mathbb Q$-Cartier, we denote by $\overline{D_{Z'}}$ its Cartier closure. Its trace on every birational model $h\colon Z''\to Z'$ is $h^*D_{Z'}$. We say that $\mathbf D$ \emph{descends to $Z'$} if $\mathbf D=\overline{D_{Z'}}$. We denote the canonical $b$-divisor over $Z$ by $\mathbf K_Z$.
\end{definition}

The discriminant and moduli parts used below arise from subadjunction \cite{Kaw98} and the canonical bundle formula \cite{FM00}. The moduli $b$-divisor and its positivity are studied in \cite{Amb05}. For the precise statement needed here, we use \cite[Theorem~1.2 and Definition~1.3]{FH23}.

\begin{proposition}\label{prop:adjunction-data}
Let $(X,\Delta)$ be a pair, let $W\subset X$ be an lc center, and assume that $(X,\Delta)$ is log canonical over the generic point of $W$. Let $\nu\colon Z\to W$ be the normalization. Then there are a discriminant $b$-divisor $\mathbf B$ and a moduli $b$-divisor $\mathbf M$ over $Z$ with the following properties.
\begin{enumerate}
\item[(i)] The trace $B_Z$ is effective, and
\[
 \nu^*\bigl((K_X+\Delta)|_W\bigr)\sim_{\mathbb Q}K_Z+B_Z+M_Z.
\]
\item[(ii)] There exist a smooth quasi-projective variety $Z'$ and a projective birational morphism $p\colon Z'\to Z$ such that $\mathbf K_Z+\mathbf B=\overline{K_{Z'}+B_{Z'}}$, $\mathbf M=\overline{M_{Z'}}$, and
$$
 K_{Z'}+B_{Z'}+M_{Z'}\sim_{\mathbb Q}p^*\nu^*\bigl((K_X+\Delta)|_W\bigr).
$$
Moreover, if $Z$ is projective, then $Z'$ is projective and $M_{Z'}$ is nef.
\item[(iii)] Let $\widetilde Z\to Z$ be a birational model and let $P$ be a prime divisor on $\widetilde Z$. For every lc place $T$ of $(X,\Delta)$ with center $W$, choose a log resolution $\pi\colon\widetilde X\to X$ on which $T$ appears and for which the induced rational map $T\dashrightarrow\widetilde Z$ is a morphism $f_T\colon T\to\widetilde Z$. If $K_T+\Delta_T=(K_{\widetilde X}+\Delta_{\widetilde X})|_T$, then
\[
 1-\mult_P\mathbf B
 =\inf_T\sup\Bigl\{\lambda\in\mathbb R\ \Bigm|\
 (T,\Delta_T+\lambda f_T^*P)
 \text{ is sub-lc over }\eta_P\Bigr\}.
\]
\end{enumerate}
\end{proposition}
\begin{proof}
The assertions follow from \cite[Theorem~1.2 and Definition~1.3]{FH23}. If $Z$ is projective, then $Z'$ is projective, and \cite[Remark~2.3]{FH23} implies that $M_{Z'}$ is nef.
\end{proof}

\subsection{Normal Hilbert coefficients}
For a Noetherian local ring $(R,\m)$, we write
$\edim R:=\dim_{R/\m}(\m/\m^2)$ for the embedding dimension of $R$.

A Noetherian local ring is \emph{analytically unramified} if its completion is reduced. For an ideal $I$, we denote its integral closure by $\overline I$.

Let $(R,\m)$ be a $d$-dimensional analytically unramified local ring with $d\geq1$. By \cite[Proposition~5.2.1 and Corollary~9.2.1]{HS06}, the graded algebra $\bigoplus_{k\geq0}\overline{\m^k}t^k$ is finite over $\bigoplus_{k\geq0}\m^kt^k$. The latter is a finitely generated $R$-algebra and hence is Noetherian, so the former is Noetherian as well. It follows that $k\mapsto\ell_R(R/\overline{\m^k})$ agrees for $k\gg0$ with a polynomial. Comparison with the Hilbert--Samuel polynomial shows that this polynomial has degree $d$ and leading coefficient $\frac{e_0(\m)}{d!}$.



\begin{definition}
Let $(R,\m)$ be a $d$-dimensional analytically unramified local ring with $d\geq1$. For $k\gg0$, write
\[
 \ell_R(R/\overline{\m^k})
 =\sum_{i=0}^{d}(-1)^i\overline{e}_i(\m)
  \binom{k+d-1-i}{d-i}.
\]
The integers $\overline{e}_i(\m)$ are the \emph{normal Hilbert coefficients} of $\m$. 
\end{definition}
Only $\ebarzero(\m)$ and $\ebarone(\m)$ will be used in this paper. The calculation on the normalized blow-up of a normal variety is expressed in terms of the first normal Hilbert coefficient $\ebarone(\m)$ (see Proposition~\ref{prop:projective-rr}), whereas Section~4 uses the Hilbert--Samuel coefficient $e_1(\m)$. These coefficients are related by $e_1(\m)\leq\ebarone(\m)$.

\begin{remark}\label{rem:normal Hilbert poly}
By \cite[Theorem~9.1.2]{HS06}, there is an integer $s_0\geq0$ such that $\m^k\subseteq\overline{\m^k}\subseteq\m^{k-s_0}$, $k\gg0$. Thus $\ebarzero(\m)=e_0(\m)$. For later use, let $W$ be a minimal lc center of an lc pair $(X,\Delta)$, and let $Z$ be the normalization of $W$. The local rings of $W$ and $Z$ at $x$ agree, and hence $e_0(\m_{W,x})=\mult_xW=\mult_xZ$. Since $\mathcal{O}_{W,x}$ is excellent and normal, its completion is normal by \cite[Lemma~15.53.6]{Sta26}. In particular, $\mathcal{O}_{W,x}$ is analytically unramified. 

When $d=1$, the normal Hilbert polynomial reads $\ell_R(R/\overline{\m^k})=e_0(\m)k-\ebarone(\m)$. When $d\geq2$, expanding the first two binomial terms, we have
\begin{equation*}
 \ell_R(R/\overline{\m^k})
 =\frac{e_0(\m)}{d!}k^d
 +\frac{(d-1)e_0(\m)-2\ebarone(\m)}{2(d-1)!}k^{d-1}
 +O(k^{d-2}).
\end{equation*}
\end{remark}

\section{The Hilbert-coefficient inequality}

\subsection{The threshold inequality}
\begin{proposition}\label{prop:threshold}
Assume that $(X,\Delta)$ is log canonical near a closed point $x$ and is not klt at $x$. Let $W:=W_x(X,\Delta)$ be positive-dimensional, let $\nu\colon Z\to W$ be the normalization, and let $\mathbf B$ be the discriminant $b$-divisor of the subadjunction $(K_X+\Delta)|_{W}$ (see Proposition~\ref{prop:adjunction-data}). Then every prime divisor $F$ over $Z$ with $c_Z(F)=x$ satisfies
\[
 1-\mult_F\mathbf B
 \geq b_x(X,\Delta)\ord_F(\m_{Z,x}).
\]
\end{proposition}
\begin{proof}
Set $b:=b_x(X,\Delta)$. Let $F$ be a prime divisor over $Z$ centered at $x$, and choose a birational model $\widetilde Z\to Z$ which extracts $F$. 

Let $T$ be an lc place of $(X,\Delta)$ with center $W$. Choose $\pi\colon\widetilde X\to X$ and $f_T\colon T\to\widetilde Z$ as in Proposition~\ref{prop:adjunction-data}\textup{(iii)}, with $P=F$. We may choose $\pi$ sufficiently high so that $\mathfrak m_x$ is principalized and the union of its divisor with $\operatorname{Supp}(\Delta_{\widetilde X})$ has simple normal crossings. Write
\[
 K_{\widetilde X}+\Delta_{\widetilde X}=\pi^*(K_X+\Delta),
 \qquad
 \mathfrak m_x\mathcal O_{\widetilde X}=\mathcal O_{\widetilde X}(-D_x).
\]
\[
\begin{tikzcd}[column sep=small]
T \arrow[rrr,hook] \arrow[d,"f_T"'] &&&
\widetilde X \arrow[d,"\pi"]\\
\widetilde Z \arrow[r]&
Z \arrow[r,"\nu"']&
W \arrow[r,hook]&X.
\end{tikzcd}
\]
Since $c_X(T)=W\ne\{x\}$, one has $\ord_T(\mathfrak m_x)=0$. Thus $T$ is not a component of $D_x$, and $D_x|_T$ is an effective Cartier divisor. Over $\eta_F$, the commutative diagram shows that
$$
\mathcal O_T(-D_x|_T)=\m_x\mathcal O_T=\m_{Z,x}\mathcal O_T
\subseteq
\mathcal O_T(-\ord_F(\m_{Z,x})f_T^*F).
$$
It follows that $D_x|_T\geq \ord_F(\m_{Z,x}) f_T^*F$ over $\eta_F$.


By the definition of $b$, one has $A_{X,\Delta}(E)-b\ord_E(\mathfrak m_x)\geq0$ for every prime divisor $E$ over $X$ whose center on $X$ contains $x$. 

Thus $(\widetilde X,\Delta_{\widetilde X}+bD_x)$ is sub-lc over a neighborhood of $x$. Since $\mult_T\Delta_{\widetilde X}=1$, let $\Delta_T:=(\Delta_{\widetilde X}-T)|_T$. By adjunction, $(T,\Delta_T+bD_x|_T)$ is sub-lc over $\eta_F$. Thus $(T,\Delta_T+b\ord_F(\m_{Z,x}) f_T^*F)$ is also sub-lc over $\eta_F$. Taking the infimum over the lc places $T$ in Proposition~\ref{prop:adjunction-data}\textup{(iii)}, we obtain $1-\mult_F\mathbf B\geq b\ord_F(\m_{Z,x})$.
\end{proof}

\begin{lemma}\label{lem:positive-b}
Let $(X,\Delta)$ be log canonical near a closed point $x$ but not klt at $x$. If $\dim W_x(X,\Delta)>0$, then $b_x(X,\Delta)>0$.
\end{lemma}

\begin{proof}
Take a log resolution $\mu\colon X'\to X$ of $(X,\Delta,\mathfrak m_x)$ and write
\[
 K_{X'}+\Delta_{X'}=\mu^*(K_X+\Delta),
 \qquad
 \mathfrak m_x\mathcal O_{X'}
 =\mathcal O_{X'}\!\left(-\sum_j r_jE_j\right),
\]
where the sum is finite, the $E_j$ are distinct prime divisors, and
$r_j>0$. Then
$$
 b_x(X,\Delta)=\min_j\frac{A_{X,\Delta}(E_j)}{r_j}.
$$
If the minimum were zero, one of the $E_j$ would be an lc place of $(X,\Delta)$ centered at $x$, so $\{x\}$ would be an lc center. This contradicts $\dim W_x(X,\Delta)>0$.
\end{proof}

\subsection{Riemann--Roch on the normalized blow-up}

\begin{proposition}\label{prop:projective-rr}
Let $Z$ be a normal variety of dimension $d\geq2$, let $x\in Z$ be a closed point, and let $\mathcal I_x\subset\mathcal O_Z$ be the ideal sheaf of $x$. Let $f\colon Y\to Z$ be the normalized blow-up of $\mathcal I_x$, write $\mathcal I_x\mathcal O_Y=\mathcal O_Y(-E)$ with $E$ effective, and let $H:=-E$. Then 
\begin{enumerate}
\item $H$ is $f$-ample and $F\cdot H^{d-1}>0$ for every prime component $F$ of $E$.
\item $-H^d=\mult_xZ$.
\item $K_Y\cdot H^{d-1}=(d-1)\mult_xZ-2\ebarone(\m_{Z,x})$.
\end{enumerate}
In particular, $\sum_F\ord_F(\m_{Z,x})(F\cdot H^{d-1})=\mult_xZ$, where $F$ runs over the prime components of $E$.
\end{proposition}

\begin{proof}
Possibly replacing $Z$ by a normal projective compactification, we may assume that $Z$ is projective. 

Let $R:=\mathcal O_{Z,x}$ and $\m:=\m_{Z,x}$. Then $R$ is analytically unramified and $\ebarone(\m)$ is defined by Remark~\ref{rem:normal Hilbert poly}. 



The normalization morphism
$Y\to\operatorname{Bl}_{\mathcal I_x}Z$
is finite, and $\mathcal O_Y(H)$ is the pullback of the tautological
$\mathcal O(1)$ on $\operatorname{Bl}_{\mathcal I_x}Z$. Hence $H$ is $f$-ample. For every prime component $F$ of $E$, the restriction $H|_F$ is ample, and hence $F\cdot H^{d-1}=(H|_F)^{d-1}>0$. 


By \cite[Lemma~1.8]{BHJ17} and \cite[Proposition~1.1.4]{HS06}, for every $k\geq1$, $f_*\mathcal O_Y(kH)=\overline{\mathcal I_x^k}$ and $\left(\overline{\mathcal I_x^k}\right)_x=\overline{\m^k}$. Since $H$ is $f$-ample, relative Serre vanishing and the Leray
spectral sequence imply that for $k\gg0$,
$\chi(Y,\mathcal O_Y(kH))= \chi(Z,\overline{\mathcal I_x^k})$.



Since $\mathcal O_Z/\overline{\mathcal I_x^k}$ is supported at $x$, taking Euler characteristics in the exact sequence
\[
 0\longrightarrow\overline{\mathcal I_x^k}
 \longrightarrow\mathcal O_Z
 \longrightarrow\mathcal O_Z/\overline{\mathcal I_x^k}
 \longrightarrow0
\]
we obtain
\begin{equation}\label{eqn: chiY(kH) lR(R/m)}
\chi(Y,\mathcal O_Y(kH))
 =\chi(Z,\mathcal O_Z)-\ell_R(R/\overline{\m^k}).
\end{equation}


Let $\mu\colon\widetilde Y\to Y$ be a projective resolution which is an isomorphism over $Y_{\mathrm{reg}}$. Since $Y$ is normal, $\mu_*\mathcal O_{\widetilde Y}=\mathcal O_Y$ and $\dim\operatorname{Supp}R^i\mu_*\mathcal O_{\widetilde Y}\leq d-2$ for every $i>0$. The projection formula for higher direct images identifies
\[
 R^i\mu_*\mathcal O_{\widetilde Y}(k\mu^*H)
 \simeq
 R^i\mu_*\mathcal O_{\widetilde Y}\otimes\mathcal O_Y(kH).
\]
By \cite[Lemma~33.33.5]{Sta26}, 
$$\chi(\widetilde Y,\mathcal O_{\widetilde Y}(k\mu^*H))
 =\chi(Y,\mathcal O_Y(kH))+\sum_{i\geq1}(-1)^i
 \chi\bigl(Y,R^i\mu_*\mathcal O_{\widetilde Y}
 \otimes\mathcal O_Y(kH)\bigr).$$
By \cite[Lemma~33.45.1]{Sta26}, for any $i>0$, $\chi\bigl(Y,R^i\mu_*\mathcal O_{\widetilde Y}
 \otimes\mathcal O_Y(kH)\bigr)$ is a numerical polynomial in
$k$ of degree at most $d-2$. Thus, by Hirzebruch--Riemann--Roch \cite[Corollary~15.2.1]{Ful98} and the projection formula,
\begin{equation}\label{eq:rr-normalized-blowup}
\begin{split}
\chi(Y,\mathcal O_Y(kH))
={}&\chi(\widetilde Y,\mathcal O_{\widetilde Y}(k\mu^*H))+O(k^{d-2})\\
={}&\frac{(\mu^*H)^d}{d!}k^d
 -\frac{K_{\widetilde Y}\cdot(\mu^*H)^{d-1}}
 {2(d-1)!}k^{d-1}+O(k^{d-2})\\
={}&\frac{H^d}{d!}k^d
 -\frac{K_Y\cdot H^{d-1}}{2(d-1)!}k^{d-1}+O(k^{d-2}).
\end{split}
\end{equation}

Combining \eqref{eqn: chiY(kH) lR(R/m)} and \eqref{eq:rr-normalized-blowup}, we obtain
\begin{equation}\label{eq:normal-hilbert-from-rr}
\ell_R(R/\overline{\m^k})
=-\frac{H^d}{d!}k^d
+\frac{K_Y\cdot H^{d-1}}{2(d-1)!}k^{d-1}
+O(k^{d-2}).
\end{equation}
Comparing \eqref{eq:normal-hilbert-from-rr} with the normal Hilbert polynomial in Remark~\ref{rem:normal Hilbert poly}, we obtain $-H^d=\mult_xZ$ and
$K_Y\cdot H^{d-1}=(d-1)\mult_xZ-2\ebarone(\m)$. Since $E=\sum_F\ord_F(\m)F$ and $H=-E$,
$\sum_F\ord_F(\m)\bigl(F\cdot H^{d-1}\bigr)=E\cdot H^{d-1}
 =-H^d=\mult_xZ.$
\end{proof}

\subsection{Boundary and moduli contributions}

\begin{lemma}\label{lem:boundary-sign}
Let $Z$ be a normal variety of dimension $d\geq2$, let $x\in Z$ be a closed point, and let $\mathcal I_x\subset\mathcal O_Z$ be the ideal sheaf of $x$. Let $f\colon Y\to Z$ be the normalized blow-up of $\mathcal I_x$, and write $\mathcal I_x\mathcal O_Y=\mathcal O_Y(H)$. If $P$ is a prime divisor on $Z$ and $\widetilde P$ is its strict transform on $Y$, then
\[ 0\geq \widetilde P\cdot H^{d-1}
 =\begin{cases}
 -\mult_xP,&x\in P,\\
 0,&x\notin P.
 \end{cases}\]
\end{lemma}

\begin{proof}
If $x\notin P$, then $\widetilde P\cdot H^{d-1}=0$. If $x\in P$, then the induced morphism $f|_{\widetilde P}\colon\widetilde P\to P$ is projective birational, and $(\mathcal I_x\mathcal O_P)\mathcal O_{\widetilde P}=\mathcal O_{\widetilde P}(H|_{\widetilde P})$ is invertible. By \cite[Theorem and Remark~(1)]{Ram73}, $-\bigl(H|_{\widetilde P}\bigr)^{d-1}=\mult_xP$. Thus $\widetilde P\cdot H^{d-1}=\bigl(H|_{\widetilde P}\bigr)^{d-1}=-\mult_xP.$
\end{proof}

\begin{lemma}\label{lem:boundary-moduli}
Assume that $(X,\Delta)$ is log canonical near a closed point $x$ and is not klt at $x$, and let $W:=W_x(X,\Delta)$. Assume that $d:=\dim W\geq2$, and let $\nu\colon Z\to W$ be the normalization. Let $\mathcal I_x\subset\mathcal O_Z$ be the ideal sheaf of $x$, and let $f\colon Y\to Z$ be its normalized blow-up. Write $\mathcal I_x\mathcal O_Y=\mathcal O_Y(-E)$ and $H:=-E$. Let $\mathbf B$ be the discriminant $b$-divisor of the subadjunction $(K_X+\Delta)|_{W}$ (see Proposition~\ref{prop:adjunction-data}), and let $B_Y$ be the trace of $\mathbf B$ on $Y$. Then
\[
 -\sum_F\mult_F B_Y\bigl(F\cdot H^{d-1}\bigr)
 \leq K_Y\cdot H^{d-1},
\]
where the sum runs over the prime components of $E$.
\end{lemma}
\begin{proof}
Possibly shrinking $X$ around $x$ and replacing $(X,\Delta)$ by its log canonical closure \cite[Corollary~1.2]{HX13}, we may assume that $X$ is projective.

By Proposition~\ref{prop:adjunction-data}(i), $B_Z$ is effective. Choose the model $p\colon Z'\to Z$ in Proposition~\ref{prop:adjunction-data}(ii). Let $\mathbf M$ be the moduli $b$-divisor of the subadjunction $(K_X+\Delta)|_{W}$ (see Proposition~\ref{prop:adjunction-data}). Since $Z$ is projective, $Z'$ is projective and $M_{Z'}$ is nef. Take a smooth projective common model
\[
\begin{tikzcd}[column sep=small]
&V \arrow[dl,"g"'] \arrow[dr,"g'"]&\\
Y \arrow[dr,"f"']&&Z' \arrow[dl,"p"]\\
&Z&
\end{tikzcd}
\]
and let $\pi:=f\circ g=p\circ g'$, $H_V:=g^*H$, and $M_V:=(g')^*M_{Z'}$. Then $M_V$ is nef. Let $B_V$ be the trace of $\mathbf B$ on $V$. By Proposition~\ref{prop:adjunction-data},
\[
 K_V+B_V+M_V\sim_{\mathbb Q}\pi^*\nu^*\bigl((K_X+\Delta)|_W\bigr).
\]
Since $\mathrm{Supp}H_V\subseteq \pi^{-1}(x)$, the projection formula implies
\[
 (K_V+B_V+M_V)\cdot H_V^{d-1}
 =\nu^*\bigl((K_X+\Delta)|_W\bigr)\cdot\pi_*H_V^{d-1}=0.
\]
Write $g^*E=\sum_j r_jF_j$, where the $F_j\subseteq\pi^{-1}(x)$
are its prime components. For each $j$, let $F_j^\nu$ denote the
normalization of $F_j$. Since $H|_{f^{-1}(x)}$ is ample,
$H_V|_{F_j^\nu}$ is nef. Since $M_V|_{F_j^\nu}$ is also nef, one has
$\bigl(M_V|_{F_j^\nu}\bigr)
 \cdot\bigl(H_V|_{F_j^\nu}\bigr)^{d-2}\geq0$. When $d=2$, this
intersection number is the degree of $M_V|_{F_j^\nu}$. Since
$H_V=-\sum_j r_jF_j$,
\[
 M_V\cdot H_V^{d-1}
 =-\sum_j r_j
 \bigl(M_V|_{F_j^\nu}\bigr)
 \cdot\bigl(H_V|_{F_j^\nu}\bigr)^{d-2}
 \leq0.
\]
Hence
\[
 (K_V+B_V)\cdot H_V^{d-1}=-M_V\cdot H_V^{d-1}\geq0.
\]
Since $g_*B_V=B_Y$ and $H_V=g^*H$, the projection formula implies
\[
 (K_Y+B_Y)\cdot H^{d-1}
 =(K_V+B_V)\cdot H_V^{d-1}\geq0.
\]
The strict transform $f_*^{-1}B_Z$ is effective, so Lemma~\ref{lem:boundary-sign} implies $(f_*^{-1}B_Z)\cdot H^{d-1}\leq0$. Since $B_Y=f_*^{-1}B_Z+\sum_F\mult_F B_YF$,
\begin{align*}
 0\leq &(K_Y+B_Y)\cdot H^{d-1}\\
 \leq &K_Y\cdot H^{d-1}
 +\sum_F\mult_F B_Y\bigl(F\cdot H^{d-1}\bigr).
\end{align*}
\end{proof}


\begin{proof}[Proof of Theorem~\ref{thm:local}]
Let $\nu\colon Z\to W$ be the normalization. By Lemma~\ref{lem:minimal-lc-center}, $\nu$ is an isomorphism near $x$. Let $\m:=\m_{Z,x}\simeq\m_{W,x}$, $d:=\dim Z=\dim W$, and $b:=b_x(X,\Delta)$. Let $\mathbf B$ be the discriminant $b$-divisor of the subadjunction $(K_X+\Delta)|_{W}$ (see Proposition~\ref{prop:adjunction-data} ).

Suppose first that $d=1$. The germ $x\in Z$ is smooth, so $\mult_xZ=\mult_xW=1$ and $\ebarone(\m)=0$. Applying Proposition~\ref{prop:threshold} to the prime divisor $x\subset Z$, we obtain $b\leq1$. Hence $2\ebarone(\m)=0\leq(1-b)\mult_xW$.

Assume $d\geq2$. Let $\mathcal I_x\subset\mathcal O_Z$ be the ideal sheaf of $x$, let $f\colon Y\to Z$ be its normalized blow-up, and write $\mathcal I_x\mathcal O_Y=\mathcal O_Y(-E)$ and $H:=-E$. Let $B_Y$ be the trace of $\mathbf B$ on $Y$. By Proposition~\ref{prop:threshold},
\[
1-\mult_F B_Y\geq b\ord_F(\m)
\]
for every prime component $F$ of $E$. Multiplying these inequalities by
$F\cdot H^{d-1}>0$, summing over the prime components of $E$, and using
Proposition~\ref{prop:projective-rr} and
Lemma~\ref{lem:boundary-moduli}, we obtain
\begin{align*}
b\mult_xZ
\leq &\sum_F\bigl(1-\mult_F B_Y\bigr)
\bigl(F\cdot H^{d-1}\bigr)\\
\leq &\mult_xZ+K_Y\cdot H^{d-1}\\
= &d\mult_xZ-2\ebarone(\m).
\end{align*}
Here the second inequality follows from
$\ord_F(\m)\geq1$,
$\sum_F\ord_F(\m)(F\cdot H^{d-1})=\mult_xZ$,
and Lemma~\ref{lem:boundary-moduli}.
Since $\mult_xZ=\mult_xW$, this proves the theorem.
\end{proof}

\section{Hilbert functions and multiplicity bounds}

We begin with an elementary estimate.

\begin{proposition}\label{prop:elementary-sequence}
Let $c$ be a positive integer and let $a\geq0$. Let $h_0,h_1,\ldots$ be a finitely supported sequence of nonnegative real numbers such that
$h_k\leq\binom{c+k-1}{k}$, $k\geq0$, and $\sum_{k\geq1}kh_k
 \leq a\sum_{k\geq0}h_k$. Then
$$ \sum_{k\geq0}h_k
 \leq\frac{(a+c)^{a+c}}{a^ac^c}.$$
\end{proposition}

\begin{proof}
If $a=0$, then $h_k=0$ for every $k\geq1$, while $h_0\leq1$, and the assertion follows. Assume that $a>0$ and that $\sum_{k\geq0}h_k>0$. For $0<z<1$, the weighted AM--GM inequality implies
\[
 \frac{\sum_{k\geq0}h_kz^k}{\sum_{k\geq0}h_k}
 \geq z^{\frac{\sum_{k\geq1}kh_k}{\sum_{k\geq0}h_k}}
 \geq z^a.
\]
On the other hand,
\[
 \sum_{k\geq0}h_kz^k
 \leq\sum_{k\geq0}\binom{c+k-1}{k}z^k
 =(1-z)^{-c}.
\]
Thus $\sum_{k\geq0}h_k\leq z^{-a}(1-z)^{-c}$. Taking $z=\frac{a}{a+c}$ proves the desired bound.
\end{proof}

We next pass to local algebra. Let $(R,\m)$ be a Noetherian local ring of dimension $d\geq1$.
The Hilbert--Samuel coefficients $e_0(\m),\ldots,e_d(\m)$ are
characterized by
\[
 \ell_R(R/\m^k)
 =\sum_{i=0}^{d}(-1)^i e_i(\m)
  \binom{k+d-1-i}{d-i}
\]
for all sufficiently large integers $k$.

Recall that an ideal $J\subseteq I$ is a \emph{reduction} of $I$ if $I^{k+1}=JI^k$ for $k\gg0$. It is a \emph{minimal reduction} if it is minimal under inclusion. If $R/\m$ is infinite, then a minimal reduction of $\m$ is generated by $d$ sufficiently general elements of $\m$.

\begin{lemma}\label{lem:artinian-reduction}
Let $(R,\m)$ be a $d$-dimensional Cohen--Macaulay local ring with $d\geq1$ and infinite residue field. Let $c:=\edim R-d$. If $J$ is a minimal reduction of $\m$, set $A:=R/J$, $\mathfrak n:=\m/J$, and $h_k:=\ell_A(\mathfrak n^k/\mathfrak n^{k+1})$. Then $e_0(\m)=\sum_{k\geq0}h_k$, $h_0=1$, $h_1=c$, and $e_1(\m)\geq\sum_{k\geq1}kh_k$. If $c>0$, then $h_k\leq\binom{c+k-1}{k}$ for every $k\geq0$.
\end{lemma}

\begin{proof}
By \cite[Theorem~8.3.6, Corollary~8.4.3,
and Proposition~8.3.7]{HS06}, minimal reductions exist and $J$ is generated by a system of parameters. Since $R$ is Cohen--Macaulay, this system of parameters is an
$R$-regular sequence, and \cite[Proposition~11.2.2]{HS06} implies $\ell_R(R/J)=e_0(\m)$. The ring $A$ is Artinian, so $\ell_A(A)=\sum_{k\geq0}h_k$. Moreover, \cite[Proposition~8.3.3(1)]{HS06} implies $J\cap\m^2=\m J$, and hence there is an exact sequence
$$
0\longrightarrow J/\m J\longrightarrow \m/\m^2\longrightarrow \mathfrak n/\mathfrak n^2\longrightarrow0.
$$
Since $J$ is generated by a system of parameters, it follows that
$h_1=\dim_{R/\m}(\mathfrak n/\mathfrak n^2)=\edim R-d=c$.
When $c>0$, taking the degree-$k$ part of the natural graded surjection
$\operatorname{Sym}_{R/\m}(\mathfrak n/\mathfrak n^2)\twoheadrightarrow
\operatorname{gr}_{\mathfrak n}(A)$ for every $k\ge 0$, we obtain
$$h_k\leq\dim_{R/\m}\operatorname{Sym}_{R/\m}^k
(\mathfrak n/\mathfrak n^2)=\binom{c+k-1}{k}.$$
Finally, by \cite[Theorem~4.7(a)]{HM97},
$$e_1(\mathfrak m) \ge \sum_{j\ge1} \ell_R\!\left(\frac{\mathfrak m^j}{J\cap\mathfrak m^j}\right)= \sum_{j\ge1}\ell_A(\mathfrak n^j) = \sum_{k\ge1}kh_k. $$
\end{proof}

\begin{corollary}\label{cor:local-algebra}
Let $(R,\m)$ be a $d$-dimensional Cohen--Macaulay local ring with infinite residue field, let $c:=\edim R-d$, and suppose that $e_1(\m)\leq a e_0(\m)$ for some $a\geq0$. If $a=0$ or $c=0$, then $e_0(\m)=1$. If $a,c>0$, then
\[
 e_0(\m)\leq\frac{(a+c)^{a+c}}{a^ac^c}.
\]
\end{corollary}

\begin{proof}
Choose a minimal reduction $J$ of $\m$, and use $A$, $\mathfrak n$, and $h_k$ as in Lemma~\ref{lem:artinian-reduction}. If $a=0$, then $\sum_{k\geq1}kh_k=0$, and hence $e_0(\m)=h_0=1$. If $c=0$, then $\mathfrak n/\mathfrak n^2=0$, so Nakayama's lemma implies $\mathfrak n=0$ and again $e_0(\m)=1$. If $a,c>0$, Proposition~\ref{prop:elementary-sequence} applies to the sequence $(h_k)_{k\geq0}$.
\end{proof}

\begin{lemma}\label{lem:ordinary-normal-e1}
Let $(R,\m)$ be a $d$-dimensional analytically unramified local ring with $d\geq1$. Then $e_1(\m)\leq\ebarone(\m)$.
\end{lemma}

\begin{proof}
Since $\m^k\subseteq\overline{\m^k}$, one has $\ell_R(R/\m^k)\geq\ell_R(R/\overline{\m^k})$. The two Hilbert polynomials have the same leading coefficient, and comparison of their coefficients of degree $d-1$ proves the assertion.
\end{proof}

We now combine Theorem~\ref{thm:local}, Corollary~\ref{cor:local-algebra}, and Lemma~\ref{lem:ordinary-normal-e1}.

\begin{proof}[Proof of Theorem~\ref{thm:multiplicity-bound}]
Let $R:=\mathcal O_{W,x}$ and $\m:=\m_{W,x}$, where $d:=\dim W$, $b:=b_x(X,\Delta)$, $a:=(d-b)/2$, and $c:=\edim R-d$. By Theorem~\ref{thm:local}, $\ebarone(\m)\leq a e_0(\m)$. 
The ring $R$ is Cohen--Macaulay by Lemma~\ref{lem:minimal-lc-center}, has residue field $\mathbb C$, and is analytically unramified by Remark~\ref{rem:normal Hilbert poly}.
Lemmas~\ref{lem:artinian-reduction} and~\ref{lem:ordinary-normal-e1} imply
\[
 0\leq e_1(\m)\leq\ebarone(\m)\leq a e_0(\m).
\]
Since $e_0(\m)>0$, $0\leq\ebarone(\m)\leq a e_0(\m)$ implies $a\geq0$, equivalently $b\leq d$. Lemma~\ref{lem:positive-b} implies $b>0$. By Corollary~\ref{cor:local-algebra}, we obtain
\[
 \mult_xW=e_0(\m)
 \leq\frac{(a+c)^{a+c}}{a^ac^c},
\]
which proves the theorem.
\end{proof}

The following corollary gives a multiplicity bound independent of the embedding dimension when
$b_x(X,\Delta)>\dim W-2$.

\begin{corollary}\label{cor:multiplicity-without-edim}
Under the assumptions and notation of
Theorem~\ref{thm:multiplicity-bound},
$$b\leq d-2+\frac{2}{\mult_xW}.$$
In particular, if $b>d-2$, then
$\mult_xW\leq\lfloor \frac{2}{b-d+2}\rfloor$.
\end{corollary}

\begin{proof}
Let $R:=\mathcal O_{W,x}$ and $\m:=\m_{W,x}$.
Lemma~\ref{lem:artinian-reduction} implies
$e_1(\m)\geq e_0(\m)-1$.
Together with Lemma~\ref{lem:ordinary-normal-e1}
and Theorem~\ref{thm:local}, we get
$$2(e_0(\m)-1)
\leq2e_1(\m)
\leq2\ebarone(\m)
\leq(d-b)e_0(\m).$$
\end{proof}


\section{Global generation and separation of points}

Recall that $\phi=(1+\sqrt5)/2$ and
$$
 C_0=\int_0^{\phi^{-2}}
 \frac{2\bigl(-z\log z-(1-z)\log(1-z)\bigr)}
 {z^{3/2}\log^2 z}\,dz.
$$
We first prove the numerical estimate needed for Proposition~\ref{prop:cutting}.

\begin{proposition}\label{prop:numerical-sum}
For every positive integer $n$, there exists $\delta_n>0$ with the following property. Let $r\geq1$, let $k$ be a positive integer with $k\leq n$, and let $n=b_0\geq b_1\geq\cdots\geq b_k=0$ be real numbers. For every $0\leq i<k$, let $d_i\geq1$, set $a_i:=(d_i-b_i)/2\geq0$, and let $0\leq c_i\leq n-d_i$. Then
\begin{equation}\label{eq:numerical-sum}
 \sum_{i=0}^{k-1}(b_i-b_{i+1})r^{1/d_i}
 \left(
 \frac{(a_i+c_i)^{a_i+c_i}}{a_i^{a_i}c_i^{c_i}}
 \right)^{1/d_i}
 \leq (C_0-\delta_n)n+3(r-1)\sqrt{2\mathrm{e}n}.
\end{equation}
\end{proposition}

\begin{proof}
We first prove \eqref{eq:numerical-sum} when $r=1$.
For $0<z\leq\phi^{-2}$, set
$\psi(z):=\frac{2\log(1-z)}{\log z}$. Then
$\lim_{z\to0^+}\psi(z)=0$, $\psi(\phi^{-2})=1$, and
$$
 \frac{1-z}{\sqrt z}\psi'(z)
 =\frac{2\bigl(-z\log z-(1-z)\log(1-z)\bigr)}
 {z^{3/2}\log^2 z}>0.
$$
Thus $\psi$ extends to a strictly increasing bijection
$[0,\phi^{-2}]\to[0,1]$ by setting $\psi(0):=0$. By the definition of
$C_0$,
$$
 C_0=\int_0^{\phi^{-2}}\frac{1-z}{\sqrt z}\psi'(z)\,dz.
$$
Since $\frac{1-z}{\sqrt z}\psi'(z)=O(z^{-1/2})$ as $z\to0^+$, this integral converges.

For $0\leq v\leq u\leq\phi^{-2}$ with $u>0$, the inequality
$-\log(1-u)\leq\frac{u}{1-u}$ implies
$$
 0\leq\bigl(\psi(u)-\psi(v)\bigr)\frac{1-u}{\sqrt u}
 \leq\psi(u)\frac{1-u}{\sqrt u}
 \leq\frac{2\sqrt u}{-\log u}\longrightarrow0
 \quad(u\to0^+).
$$
Thus $(u,v)\mapsto\bigl(\psi(u)-\psi(v)\bigr)\frac{1-u}{\sqrt u}$ extends continuously to $(0,0)$ with value $0$.

Let $\phi^{-2}=z_0\geq z_1\geq\cdots\geq z_n=0$. Since
$z\mapsto\frac{1-z}{\sqrt z}$ is strictly decreasing,
\begin{equation}\label{eq:numerical-integral-comparison}
\begin{split}
 \sum_{j=0}^{n-1}
 \bigl(\psi(z_j)-\psi(z_{j+1})\bigr)
 \frac{1-z_j}{\sqrt{z_j}}
 ={}&\sum_{j=0}^{n-1}\int_{z_{j+1}}^{z_j}
 \frac{1-z_j}{\sqrt{z_j}}\psi'(z)\,dz\\
 <{}&\int_0^{\phi^{-2}}
 \frac{1-z}{\sqrt z}\psi'(z)\,dz=C_0.
\end{split}
\end{equation}
Here a summand is defined to be $0$ when $z_j=0$. The inequality is strict because $z_0>z_n$, the function
$z\mapsto\frac{1-z}{\sqrt z}$ is strictly decreasing, and $\psi'(z)>0$ for
$0<z\leq\phi^{-2}$.

By the continuous extension above, the left-hand side of \eqref{eq:numerical-integral-comparison} is continuous on the compact simplex $\phi^{-2}=z_0\geq\cdots\geq z_n=0$. Hence there exists $\delta_n>0$, depending only on $n$, such that the left-hand side of \eqref{eq:numerical-integral-comparison} is at most $C_0-\delta_n$.

If $b_j=0$ for some $j<k$, then $b_i=0$ for every $i\geq j$, and all summands with $i\geq j$ vanish. Replacing $k$ by the first such $j$, we may assume that $b_i>0$ for $i<k$. For $0\leq i<k$, let $z_i\in(0,\phi^{-2}]$ be determined by
$\psi(z_i)=\frac{b_i}{n}$, and set $z_k:=0$. Then
$\phi^{-2}=z_0\geq z_1\geq\cdots\geq z_k=0$ and
$1-z_i=z_i^{\frac{b_i}{2n}}$ for $i<k$. As in the proof of Proposition~\ref{prop:elementary-sequence}, for every $0<u<1$,
$$
 \frac{(a_i+c_i)^{a_i+c_i}}{a_i^{a_i}c_i^{c_i}}
 \leq u^{-a_i}(1-u)^{-c_i}.
$$
Taking $u=z_i$, and using $c_i\leq n-d_i$ and
$a_i+\frac{b_i(n-d_i)}{2n}=\frac{d_i}{2}\left(1-\frac{b_i}{n}\right)$, we obtain
\begin{equation}\label{eq:numerical-pointwise}
 \left(\frac{(a_i+c_i)^{a_i+c_i}}{a_i^{a_i}c_i^{c_i}}\right)^{1/d_i}
 \leq z_i^{-\frac{1}{d_i}\left(a_i+\frac{b_i c_i}{2n}\right)}
 \leq z_i^{-\frac12\left(1-\frac{b_i}{n}\right)}
 =\frac{1-z_i}{\sqrt{z_i}}.
\end{equation}
If $k<n$, set $z_{k+1}=\cdots=z_n:=0$. Since
$b_i-b_{i+1}=n\bigl(\psi(z_i)-\psi(z_{i+1})\bigr)$,
the choice of $\delta_n$ and \eqref{eq:numerical-pointwise} imply
$$
 \sum_{i=0}^{k-1}(b_i-b_{i+1})
 \left(\frac{(a_i+c_i)^{a_i+c_i}}{a_i^{a_i}c_i^{c_i}}\right)^{1/d_i}
 \leq n(C_0-\delta_n).
$$
This proves \eqref{eq:numerical-sum} when $r=1$.

Now let $r\geq1$. If $a_i,c_i>0$, then $d_i<n$. Since $\frac{(a+c)^{a+c}}{a^ac^c}$ is nondecreasing in both variables $a,c$, and $a_i\leq\frac{d_i}{2},c_i\leq n-d_i$, one has
$$\left(\frac{(a_i+c_i)^{a_i+c_i}}{a_i^{a_i}c_i^{c_i}}\right)^{1/d_i}
\leq \left(1+\frac{d_i}{2(n-d_i)}\right)^{\frac{n-d_i}{d_i}}
\left(\frac{2n-d_i}{d_i}\right)^{1/2}
\leq \sqrt{\frac{2\mathrm{e}n}{d_i}}.$$
The last inequality follows from $1+u\leq\mathrm{e}^u$ and $2n-d_i\leq2n$. The above inequality also holds when $a_ic_i=0$.

By Bernoulli's inequality, $r^u-1\leq(r-1)u$ for $0\leq u\leq1$. Since $d_i\geq\max\{1,t\}$ for $t\in[b_{i+1},b_i]$, we obtain
\begin{align*}
&\sum_{i=0}^{k-1}(b_i-b_{i+1})\bigl(r^{1/d_i}-1\bigr)
\left(\frac{(a_i+c_i)^{a_i+c_i}}{a_i^{a_i}c_i^{c_i}}\right)^{1/d_i}\\
\leq &(r-1)\sqrt{2\mathrm{e}n}
\sum_{i=0}^{k-1}\frac{b_i-b_{i+1}}{d_i^{3/2}}\\
\leq &(r-1)\sqrt{2\mathrm{e}n}\left(1+\int_1^n\frac{dt}{t^{3/2}}\right)
\leq3(r-1)\sqrt{2\mathrm{e}n}.
\end{align*}
Combining the above inequality with the case $r=1$ proves \eqref{eq:numerical-sum}.
\end{proof}




\begin{lemma}\label{lem:one-step-cutting}
Let $(X,\Delta)$ be a pair with $X$ smooth projective, let $L$ be an ample $\mathbb Q$-divisor, and let $S\subset X$ be a nonempty finite set of closed points. If $\Delta=0$, set $W:=X$. Otherwise, assume that $(X,\Delta)$ is log canonical near every point of $S$ but not klt at any point of $S$, and that $W:=W_y(X,\Delta)$ is independent of $y\in S$. Suppose that $d:=\dim W>0$. Then for every $\varepsilon>0$, there exist a nonempty subset $S'\subseteq S$, a subvariety $W'\subsetneq W$, an effective $\mathbb Q$-divisor $D\sim_{\mathbb Q}L$, and a rational number $t\geq0$ such that, for $\Delta':=\Delta+tD$, the pair $(X,\Delta')$ is log canonical near every point of $S'$, one has $W_y(X,\Delta')=W'$ for every $y\in S'$, and
$$
b_y(X,\Delta')\leq b_y(X,\Delta)
-t\left(\frac{L^d\cdot W}{|S|\mult_yW}\right)^{1/d}+\varepsilon
$$
for every $y\in S'$. Moreover, every point of $S\setminus S'$ is contained in a non-klt center of $(X,\Delta')$ which contains no point of $S'$.
\end{lemma}

\begin{proof}
The case $\Delta=0$ follows from \cite[Remark~3.3]{GL24}.
The other case follows from \cite[Proposition~3.2]{GL24} with both finite sets equal to $S$. 
\end{proof}

For a big Cartier divisor $L$, let $\mathbf B_+(L)$ denote its augmented base locus. Thus $\mathbf B_+(L)=\mathbf B(L-\varepsilon A)$ for an ample Cartier divisor $A$ and every sufficiently small positive rational $\varepsilon$, where $\mathbf B$ denotes the stable base locus \cite[Proposition~1.5]{ELMNP06}. If $L$ is ample, then $\mathbf B_+(L)=\emptyset$.

\begin{proposition}\label{prop:cutting}
Let $X$ be a smooth projective variety of dimension $n$, let $L$ be a nef and big Cartier divisor, and let $S\subset X\setminus\mathbf B_+(L)$ be a set of $r\geq1$ distinct closed points. Let $\delta_n>0$ be as in Proposition~\ref{prop:numerical-sum}. Then there exist a point $x\in S$, a rational number $s<(C_0-\delta_n/8)n+3(r-1)\sqrt{2\mathrm{e}n}$, and an effective $\mathbb Q$-divisor $\Delta\sim_{\mathbb Q}sL$ such that $(X,\Delta)$ is log canonical near $x$, the point $\{x\}$ is the unique lc center through $x$, and $(X,\Delta)$ is not klt at any point of $S$.
\end{proposition}

\begin{proof}
We first carry out the construction for an ample $\mathbb Q$-divisor $L$ satisfying $L^{\dim W}\cdot W\geq1$ for every positive-dimensional subvariety $W\subseteq X$ with $W\cap S\neq\emptyset$. Set $\Delta_0:=0$, $S_0:=S$, $W_0:=X$, and $d_0:=n$.

Suppose that $\Delta_i,S_i,W_i$, and $d_i$ have been constructed and that $d_i>0$. Since $S_i$ is finite, apply Lemma~\ref{lem:one-step-cutting} with $S=S_i$ and $\varepsilon>0$ sufficiently small. There exist a nonempty subset $S_{i+1}\subseteq S_i$, an effective $\mathbb Q$-divisor $D_{i+1}\sim_{\mathbb Q}L$, and a rational number $t_{i+1}\geq0$ such that, for $\Delta_{i+1}:=\Delta_i+t_{i+1}D_{i+1}$, the pair $(X,\Delta_{i+1})$ is log canonical near every point of $S_{i+1}$, one has $W_y(X,\Delta_{i+1})=W_{i+1}\subsetneq W_i$ for every $y\in S_{i+1}$, and
$$
 t_{i+1}<\bigl(b_y(X,\Delta_i)-b_y(X,\Delta_{i+1})\bigr)
 \left(\frac{r\mult_yW_i}{L^{d_i}\cdot W_i}\right)^{1/d_i}
 +\frac{\delta_n}{2}
$$
for every $y\in S_{i+1}$, where we use $|S_i|\leq r$. Set $d_{i+1}:=\dim W_{i+1}$. The lemma and $\Delta_{i+1}\geq\Delta_i$ imply inductively that every point of $S\setminus S_{i+1}$ is contained in a non-klt center of $(X,\Delta_{i+1})$ which contains no point of $S_{i+1}$.

Since $d_{i+1}<d_i$, the construction terminates with $n=d_0>d_1>\cdots>d_{k-1}>d_k=0$ for some $k\leq n$. Choose $x\in S_k$. Then $W_k=S_k=\{x\}$. Set $b_i:=b_x(X,\Delta_i)$ for $0\leq i\leq k$. Since $\{x\}$ is an lc center of $(X,\Delta_k)$, we have $b_k=0$. Moreover, $\Delta_{i+1}\geq\Delta_i$ implies $b_{i+1}\leq b_i$. Thus $n=b_0\geq b_1\geq\cdots\geq b_k=0$.



For $0\leq i<k$, let $a_i:=(d_i-b_i)/2$ and $c_i:=\edim\mathcal O_{W_i,x}-d_i$. By Theorem~\ref{thm:multiplicity-bound}, $a_i\geq0$ for $i>0$, while $a_0=c_0=0$. Since $\mathcal O_{W_i,x}$ is a quotient of the regular local ring $\mathcal O_{X,x}$, $0\leq c_i\leq n-d_i$. By Theorem~\ref{thm:multiplicity-bound} and Proposition~\ref{prop:numerical-sum}, we obtain
\begin{align*}
 \sum_{i=0}^{k-1}t_{i+1}
 < &\sum_{i=0}^{k-1}(b_i-b_{i+1})r^{1/d_i}
 \left(
 \frac{(a_i+c_i)^{a_i+c_i}}{a_i^{a_i}c_i^{c_i}}
 \right)^{1/d_i}+\frac{k\delta_n}{2}\\
 \leq &(C_0-\delta_n)n+3(r-1)\sqrt{2\mathrm{e}n}+\frac{k\delta_n}{2}\\
 \leq &\left(C_0-\frac{\delta_n}{2}\right)n+3(r-1)\sqrt{2\mathrm{e}n}.
\end{align*}

Every point of $S\setminus\{x\}$ is contained in a non-klt center of $(X,\Delta_k)$ which does not contain $x$. Apply \cite[Lemma~2.8]{GL24} with $S$ as above and $D=L$. There exists a positive rational number $q$ such that, for every sufficiently small rational $\lambda>0$, there is an effective $\mathbb Q$-divisor $D_\lambda\sim_{\mathbb Q}qL$ such that, for $\Delta:=(1-\lambda)\Delta_k+\lambda D_\lambda$, the pair $(X,\Delta)$ is log canonical near $x$, the point $\{x\}$ is the unique lc center through $x$, and $(X,\Delta)$ is not klt at any point of $S$. Moreover, $\Delta\sim_{\mathbb Q}sL$, where $s:=(1-\lambda)\sum_{i=0}^{k-1}t_{i+1}+\lambda q$. Choose $\lambda$ such that $s<(C_0-\delta_n/4)n+3(r-1)\sqrt{2\mathrm{e}n}$. 

Now assume that $L$ is a nef and big Cartier divisor and $S\subset X\setminus\mathbf B_+(L)$. By \cite[Proposition~1.5]{ELMNP06}, there exist an ample $\mathbb Q$-divisor $A$ and an effective $\mathbb Q$-divisor $E$ such that $L\sim_{\mathbb Q}A+E$ and $S\cap\operatorname{Supp}E=\emptyset$.

If $W$ is a positive-dimensional subvariety of $X$ with
$W\cap S\neq\emptyset$, then
$W\not\subseteq\operatorname{Supp}E$.
Thus $L|_W$ is nef and big, and $L^{\dim W}\cdot W\geq1$.
Thus for every positive rational number $t$, the divisor $L+tA$
is ample and $(L+tA)^{\dim W}\cdot W\geq1$ for every such $W$.

Choose a positive rational number $t$ such that
$$
(1+t)\left((C_0-\delta_n/4)n+3(r-1)\sqrt{2\mathrm{e}n}\right)
<(C_0-\delta_n/8)n+3(r-1)\sqrt{2\mathrm{e}n}.
$$
By the ample case applied to $L+tA$, there exist a point $x\in S$ and an effective $\mathbb Q$-divisor $\Delta'\sim_{\mathbb Q}s'(L+tA)$, where $s'<(C_0-\delta_n/4)n+3(r-1)\sqrt{2\mathrm{e}n}$, such that $(X,\Delta')$ is log canonical near $x$, $\{x\}$ is the unique lc center through $x$, and $(X,\Delta')$ is not klt at any point of $S$. Set $s:=(1+t)s'$ and $\Delta:=\Delta'+s'tE$. Then $\Delta=\Delta'$ near $S$ and $\Delta\sim_{\mathbb Q}sL$. By the choice of $t$, $s<(C_0-\delta_n/8)n+3(r-1)\sqrt{2\mathrm{e}n}$. This proves the proposition.
\end{proof}


\begin{proof}[Proof of Theorem~\ref{thm:main}]
Fix a closed point $x\in X$. By Proposition~\ref{prop:cutting} with $S=\{x\}$, there is an effective $\mathbb Q$-divisor $\Delta\sim_{\mathbb Q}sL$, where $s<C_0n\leq m$, such that $(X,\Delta)$ is log canonical near $x$ and $\{x\}$ is the unique lc center through $x$. Let $\mathcal J:=\mathcal J(X,\Delta)$ and consider
\begin{equation}\label{eq:nadel-exact-sequence}
 0\longrightarrow\mathcal O_X(K_X+mL)\otimes\mathcal J
 \longrightarrow\mathcal O_X(K_X+mL)
 \longrightarrow\mathcal O_X(K_X+mL)\otimes\mathcal O_X/\mathcal J
 \longrightarrow0.
\end{equation}
Since $\{x\}$ is the unique lc center through $x$, the point $x$ is an isolated component of $\operatorname{Supp}(\mathcal O_X/\mathcal J)$. Hence $\mathcal O_X(K_X+mL)\otimes\mathcal O_X/\mathcal J$ has a direct summand supported at $x$, and this summand surjects onto $\mathcal O_X(K_X+mL)|_x$ (cf. \cite[Proof of Theorem~5.8]{Kol97}, \cite[Proof of Theorem~5.2]{GL24}). Since $(m-s)L$ is ample, by Nadel vanishing \cite[Theorem~9.4.17]{Laz04}, we have
$$
 H^1\bigl(X,\mathcal O_X(K_X+mL)\otimes\mathcal J\bigr)=0.
$$
Therefore, taking cohomology of \eqref{eq:nadel-exact-sequence}, we obtain a surjection
$$
 H^0\bigl(X,\mathcal O_X(K_X+mL)\bigr)
 \twoheadrightarrow\mathcal O_X(K_X+mL)|_x.
$$
Thus $x$ is not a base point of $K_X+mL$. Since this holds for every closed point $x\in X$, $K_X+mL$ is globally generated.
\end{proof}

\begin{proof}[Proof of Corollary~\ref{cor:2n}]
This follows from Theorem~\ref{thm:main}, since $C_0<2$.
\end{proof}

We say that a line bundle $\mathcal L$ on $X$ \emph{separates $r$ points} if the restriction map $H^0(X,\mathcal L)\to\bigoplus_{x\in S}\mathcal L|_x$ is surjective for every set $S\subset X$ of $r$ distinct closed points.


\begin{theorem}\label{thm:point-separation}
Let $X$ be a smooth projective variety of dimension $n$, let $L$ be a nef and big Cartier divisor, and let $N$ be a nef Cartier divisor. Let $r$ be a positive integer. Then $K_X+mL+N$ separates any $r$ distinct closed points in $X\setminus\mathbf B_+(L)$ for every integer $m\geq\lceil C_0n+3(r-1)\sqrt{2\mathrm{e}n}\rceil$. In particular, $\operatorname{Bs}|K_X+mL+N|\subseteq\mathbf B_+(L)$ for every integer $m\geq\lceil C_0n\rceil$.
\end{theorem}

\begin{proof}
We argue by induction on $r$. Let $S\subset X\setminus\mathbf B_+(L)$ be a set of $r$ distinct closed points. For every nonempty subset $T\subseteq S$, let
\begin{equation}\label{eq:point-separation-restriction}
\rho_T\colon H^0(X,\mathcal O_X(K_X+mL+N))
\longrightarrow\bigoplus_{y\in T}\mathcal O_X(K_X+mL+N)|_y
\end{equation}
be the restriction map. We show that $\rho_S$ is surjective.
By Proposition~\ref{prop:cutting}, there exist a point $x\in S$ and an effective $\mathbb Q$-divisor $\Delta\sim_{\mathbb Q}sL$, where $s<C_0n+3(r-1)\sqrt{2\mathrm{e}n}\leq m$, such that $(X,\Delta)$ is log canonical near $x$, $\{x\}$ is the unique lc center through $x$, and $(X,\Delta)$ is not klt at any point of $S$.


Set $\mathcal J:=\mathcal J(X,\Delta)$. Since $mL+N-\Delta\sim_{\mathbb Q}(m-s)L+N$ is nef and big, Nadel vanishing \cite[Theorem~9.4.17]{Laz04} implies
$$
H^0(X,\mathcal O_X(K_X+mL+N))
\twoheadrightarrow
H^0(X,\mathcal O_X(K_X+mL+N)\otimes\mathcal O_X/\mathcal J).
$$
Arguing as in the proof of Theorem~\ref{thm:main} and using $\mathcal J_y\subseteq\m_y$ for every $y\in S\setminus\{x\}$, we obtain a section $\sigma$ of $K_X+mL+N$ which does not vanish at $x$ and vanishes on $S\setminus\{x\}$. This proves the case $r=1$. 

For $r>1$, $\rho_{S\setminus\{x\}}$ is surjective by induction. Together with $\sigma|_x\neq0$ and $\sigma|_y=0$ for every $y\in S\setminus\{x\}$, this implies that $\rho_S$ is surjective.
\end{proof}

\begin{corollary}\label{cor:adjoint-morphisms}
Let $X$ be a smooth projective variety of dimension $n$ and let $L$ be an ample Cartier divisor. Then for every integer $m\geq\lceil C_0n+3\sqrt{2\mathrm{e}n}\rceil$, the morphism defined by $|K_X+mL|$ is finite and is a homeomorphism onto its image.
\end{corollary}

\begin{proof}
By Theorem~\ref{thm:main}, $K_X+mL$ is globally generated. Let $f\colon X\to Y$ be the morphism defined by $|K_X+mL|$, where $Y=f(X)$. By Theorem~\ref{thm:point-separation} with $r=2$ and $N=0$, $f$ is injective on closed points. Since $f$ is projective with finite fibers, it is finite. Hence $f$ is a homeomorphism onto its image (cf. \cite[Proof of Proposition~II.7.3]{Har77}).
\end{proof}

We next prove a birationality criterion for adjoint linear systems.

\begin{proposition}\label{prop:adjoint-birationality}
Let $Z$ be a normal projective variety of dimension $n$, let $L$ be a nef and big Cartier divisor, and let $D$ be a $\mathbb Q$-Cartier $\mathbb Q$-divisor. If $D-(C_0n+3\sqrt{2\mathrm{e}n})L$ is pseudo-effective, then $|K_Z+\lceil D\rceil|$ defines a birational map.
\end{proposition}

\begin{proof}
Let $\pi\colon X\to Z$ be a resolution which is an isomorphism over $Z_{\mathrm{sm}}$. Set $D':=K_X+\lceil\pi^*D\rceil$. Let $\delta_n>0$ be as in Proposition~\ref{prop:numerical-sum}. Choose a rational number $q$ such that $(C_0-\delta_n/16)n+3\sqrt{2\mathrm{e}n}<q<C_0n+3\sqrt{2\mathrm{e}n}$. Since $D-qL$ is big, there exists an effective $\mathbb Q$-divisor $F\sim_{\mathbb Q}\pi^*(D-qL)$. Set
$U:=X\setminus\left(\mathbf B_+(\pi^*L)\cup\operatorname{Exc}(\pi)
\cup\operatorname{Supp}(F+\lceil\pi^*D\rceil-\pi^*D)\right).$

Let $S\subset U$ be a set of one or two distinct closed points. For every nonempty subset $T\subseteq S$, let
\begin{equation}\label{eq:adjoint-birationality-restriction}
\rho_T\colon H^0(X,\mathcal O_X(D'))
\longrightarrow\bigoplus_{y\in T}\mathcal O_X(D')|_y
\end{equation}
be the restriction map. We prove by induction on $|S|$ that $\rho_S$ is surjective.

By Proposition~\ref{prop:cutting} applied to $\pi^*L$, there exist a point $x\in S$ and an effective $\mathbb Q$-divisor $\Delta'\sim_{\mathbb Q}s\pi^*L$, where $s<(C_0-\delta_n/8)n+3\sqrt{2\mathrm{e}n}<q$, such that $(X,\Delta')$ is log canonical near $x$, $\{x\}$ is the unique lc center through $x$, and $(X,\Delta')$ is not klt at any point of $S$.

Set $\Delta:=\Delta'+F+\lceil\pi^*D\rceil-\pi^*D$ and $\mathcal J:=\mathcal J(X,\Delta)$. Then $\Delta\geq0$, and $\Delta=\Delta'$ on $U$. Moreover, $D'-K_X-\Delta\sim_{\mathbb Q}(q-s)\pi^*L$ is nef and big. Nadel vanishing \cite[Theorem~9.4.17]{Laz04} implies
$$
H^0(X,\mathcal O_X(D'))
\twoheadrightarrow
H^0(X,\mathcal O_X(D')\otimes\mathcal O_X/\mathcal J).
$$
Since $x$ is an isolated point of $\operatorname{Supp}(\mathcal O_X/\mathcal J)$, we may lift a section of $\mathcal O_X(D')\otimes\mathcal O_X/\mathcal J$ supported at $x$ and nonzero at $x$ to a section $\sigma$ of $D'$. Then $\sigma|_x\neq0$, and $\sigma|_y=0$ for every $y\in S\setminus\{x\}$ since $\mathcal J_y\subseteq\m_y$.

If $|S|=1$, then $\rho_S$ is surjective since $\sigma|_x\neq0$. If $|S|=2$, then $\rho_{S\setminus\{x\}}$ is surjective by the case $|S|=1$, and $\rho_S$ is surjective since $\rho_{S\setminus\{x\}}(\sigma)=0$ and $\sigma|_x\neq0$. Thus $|D'|$ has no base points on $U$ and separates any two distinct closed points of $U$. Hence $|D'|$ defines a birational map.

Since $Z$ is normal and $\pi_*D'=K_Z+\lceil D\rceil$, we have a natural inclusion
$$
H^0(X,\mathcal O_X(D'))
\hookrightarrow H^0(Z,\mathcal O_Z(K_Z+\lceil D\rceil)).
$$
Thus $|K_Z+\lceil D\rceil|$ defines a birational map.
\end{proof}

\begin{corollary}\label{cor:polarized-birationality}
Let $(Z,B)$ be a projective pair of dimension $n$, and let $L,N$ be $\mathbb Q$-Cartier $\mathbb Q$-divisors on $Z$ such that $L$ is nef and big and $N$ is pseudo-effective. Let $I$ be a positive integer such that $IL$ is Cartier. Assume that $aL-(K_Z+B)$ is pseudo-effective for some rational number $a$. Then $|\lceil mL+N\rceil|$ defines a birational map for every positive integer $m\geq\lceil I(C_0n+3\sqrt{2\mathrm{e}n})+a\rceil$.
\end{corollary}

\begin{proof}
Let $D:=mL+N-(K_Z+B)$. Since $(m-a)/I\geq C_0n+3\sqrt{2\mathrm{e}n}$ and $D-\frac{m-a}{I}(IL)=aL-(K_Z+B)+N$ is pseudo-effective, $|K_Z+\lceil D\rceil|$ defines a birational map by Proposition~\ref{prop:adjoint-birationality} applied to $D$ and $IL$. Since $B\geq0$, we have $K_Z+\lceil D\rceil=\lceil mL+N-B\rceil\leq\lceil mL+N\rceil$. This proves the corollary.
\end{proof}

\begin{remark}
If $-(K_Z+B)$ is pseudo-effective, $L$ is a nef and big Cartier divisor, and $N$ is a pseudo-effective Cartier divisor, then $|mL+N|$ defines a birational map for every integer $m\geq\lceil C_0n+3\sqrt{2\mathrm{e}n}\rceil$ by Corollary~\ref{cor:polarized-birationality} with $a=0$ and $I=1$. This applies, in particular, when $K_Z+B\equiv0$.
\end{remark}


\begin{corollary}\label{cor:pluricanonical-systems}
Let $(Z,B)$ be a projective pair of dimension $n$ and let $N$ be a pseudo-effective Cartier divisor. Let $I$ be a positive integer such that $I(K_Z+B)$ is Cartier. Then the following statements hold.
\begin{enumerate}
\item[(1)] If $K_Z+B$ is nef and big, then $|\lceil m(K_Z+B)\rceil+N|$ defines a birational map for every integer $m\geq1+\lceil I(C_0n+3\sqrt{2\mathrm{e}n})\rceil$.
\item[(2)] If $-(K_Z+B)$ is nef and big, then $|-\lfloor m(K_Z+B)\rfloor+N|$ defines a birational map for every integer $m\geq\lceil I(C_0n+3\sqrt{2\mathrm{e}n})\rceil-1$.
\end{enumerate}
\end{corollary}

\begin{proof}
Apply Corollary~\ref{cor:polarized-birationality} with $L=K_Z+B$ and $a=1$ in (1), and with $L=-(K_Z+B)$ and $a=-1$ in (2).
\end{proof}

For $B=0$, the two systems in Corollary~\ref{cor:pluricanonical-systems} are $|mK_Z+N|$ and $|-mK_Z+N|$, respectively. For related anti-pluricanonical results, see \cite{Bir19,Liu19,JZ24}.

\begin{proof}[Proof of Theorem~\ref{thm:pluricanonical-birationality}]
This follows from Corollary~\ref{cor:pluricanonical-systems}\textup{(1)} with $B=N=0$ and $I=1$.
\end{proof}

It is interesting to ask the following questions which are related to Fujita's conjecture and \cite[Conjecture~0.1]{Cat22}.
\begin{question} 
Let $Z$ be a normal projective variety of dimension $n$, let $L$ be a nef and big Cartier divisor, and let $D$ be a $\mathbb Q$-Cartier $\mathbb Q$-divisor. If $D-(n+1)L$ is big, does $|K_Z+\lceil D\rceil|$ define a birational map? 
\end{question}

\begin{question}
Let $Z$ be a normal projective variety of dimension $n$. Assume that $K_Z$ is Cartier, nef and big. Then the
following statements hold.
\begin{enumerate}
\item[(1)] The linear system $|mK_Z|$ defines a birational map for
every integer $m\geq n+3$.
\item[(2)] If $K_Z^n\geq2$, then $|mK_Z|$ defines a birational map
for every integer $m\geq n+2$.
\end{enumerate}
\end{question}

\appendix

\section{Refined multiplicity bounds}
\label{sec:discrete-refinements}


The result in this appendix is not needed for the proof of Theorem~\ref{thm:main}.

\begin{proposition}\label{prop:discrete-artinian}
Let $(R,\m)$ be a $d$-dimensional Cohen--Macaulay local ring with infinite residue field, where $d\geq1$, and let $c:=\edim R-d$. Suppose that $e_1(\m)\leq a e_0(\m)$ for some $a\geq0$. If $c=0$, then $e_0(\m)=1$. If $c>0$, then $a\geq\frac{c}{c+1}$, and
\begin{equation}\label{eq:exact-artinian-multiplicity}
 e_0(\m)\leq
 \left\lfloor
 \frac{1}{s+1-a}\binom{c+s+1}{c+1}
 \right\rfloor,
\end{equation}
where $s:=\left\lfloor\frac{(c+1)a}{c}\right\rfloor$. 
\end{proposition}

\begin{proof}
Choose a minimal reduction $J$ of $\m$, and let $A=R/J$, $\mathfrak n=\m/J$, and $h_k$ be as in Lemma~\ref{lem:artinian-reduction}. Then $e_0(\m)=\sum_{k\geq0}h_k$, $h_0=1$, $h_1=c$, and
$\sum_{k\geq1}kh_k\leq e_1(\m)\leq a e_0(\m)$. If $c=0$, then $\mathfrak n/\mathfrak n^2=0$, so Nakayama's lemma implies $\mathfrak n=0$ and $e_0(\m)=1$. We may therefore assume that $c>0$. By Lemma~\ref{lem:artinian-reduction}, $0\leq h_k\leq\binom{c+k-1}{k}$ for every $k\geq0$. Since $h_0=1$ and $h_1=c$,
$$
 0\leq\sum_{k\geq2}\bigl((c+1)k-c\bigr)h_k
 =(c+1)\sum_{k\geq1}kh_k-ce_0(\m)
 \leq\bigl((c+1)a-c\bigr)e_0(\m).
$$
Since $e_0(\m)>0$, one has $a\geq\frac{c}{c+1}$.

Using $kh_k\geq(s+1)h_k$ for $k\geq s+1$ and
$h_k\leq\binom{c+k-1}{k}$ for $0\leq k\leq s$, we obtain
\begin{equation}\label{eq:discrete-first-moment-lower}
 \sum_{k\geq1}kh_k
 \geq (s+1)e_0(\m)-\sum_{k=0}^{s}(s+1-k)h_k
 \geq (s+1)e_0(\m)-\binom{c+s+1}{c+1}.
\end{equation}
Indeed,
$$\sum_{k=0}^{s}(s+1-k)\binom{c+k-1}{k}
 = \sum_{j=0}^{s}\sum_{k=0}^{j}\binom{c+k-1}{k}
 = \sum_{j=0}^{s}\binom{c+j}{c}
 =\binom{c+s+1}{c+1}.$$
Combining \eqref{eq:discrete-first-moment-lower} with
$\sum_{k\geq1}kh_k\leq ae_0(\m)$, we obtain
$$
 (s+1-a)e_0(\m)\leq\binom{c+s+1}{c+1}.
$$
Since $s+1>\frac{(c+1)a}{c}>a$, dividing by $s+1-a$ and taking the floor proves \eqref{eq:exact-artinian-multiplicity}.
\end{proof}

\begin{corollary}\label{cor:discrete-lc-center}
Assume that $(X,\Delta)$ is log canonical near a closed point $x$ but not klt at $x$, and $W:=W_x(X,\Delta)$ is positive-dimensional. Let $d:=\dim W$, $b:=b_x(X,\Delta)$, $a:=(d-b)/2$, and $c:=\edim\mathcal O_{W,x}-d$. If $c=0$, then $\mult_xW=1$. Assume that $c>0$, and let $s:=\left\lfloor\frac{(c+1)a}{c}\right\rfloor$. Then
$$
 \mult_xW\leq
 \left\lfloor
 \frac{1}{s+1-a}\binom{c+s+1}{c+1}
 \right\rfloor.
$$
\end{corollary}

\begin{proof}
Let $R:=\mathcal O_{W,x}$ and $\m:=\m_{W,x}$. By Theorem~\ref{thm:multiplicity-bound}, $0<b\leq d$, and hence $a\geq0$. By Lemma~\ref{lem:minimal-lc-center}, $R$ is Cohen--Macaulay, and its residue field is $\mathbb C$. Theorem~\ref{thm:local} and Lemma~\ref{lem:ordinary-normal-e1} imply
$e_1(\m)\leq\ebarone(\m)\leq a\mult_xW=ae_0(\m)$. Applying Proposition~\ref{prop:discrete-artinian} to $R$ proves the corollary.
\end{proof}


\begin{remark}\label{rem:discrete-lc-center} With the notation of Corollary~\ref{cor:discrete-lc-center},
assume that $c>0$.
An elementary calculation shows
$$ \left\lfloor
 \frac{1}{s+1-a}\binom{c+s+1}{c+1}
 \right\rfloor\le \min\{\frac{(a+c)^{a+c}}{a^ac^c},\binom{d+c-\lceil b\rceil}{c}\}.$$
The inequality can be strict. Let $X=\mathbb P^5$, let $x=[1:0:\cdots:0]$, and choose quadratic forms $q_1,q_2$ in the last five homogeneous coordinates whose associated quadrics in $\mathbb P^4$ are smooth and meet transversely. Then $S:=\{q_1=q_2=0\}\subset\mathbb P^4$ is smooth. Let $D_i\subset X$ be the quadric cone defined by $q_i=0$, let $\Delta:=D_1+D_2$, and let $W:=D_1\cap D_2$. Thus $W$ is the projective cone over $S$ and has dimension three.

Let $\pi\colon\widetilde X\to X$ be the blow-up of $x$, with exceptional divisor $E$, and let $\widetilde D_i$ be the strict transform of $D_i$. The divisors $E,\widetilde D_1,\widetilde D_2$ have simple normal crossings near $E$, and
$$
 K_{\widetilde X}+\widetilde D_1+\widetilde D_2
 =\pi^*(K_X+D_1+D_2).
$$
Hence $(X,\Delta)$ is log canonical near $x$, and $W$ is its minimal lc center through $x$. Moreover, $\m_x\mathcal O_{\widetilde X}=\mathcal O_{\widetilde X}(-E)$ and $(\widetilde X,\widetilde D_1+\widetilde D_2+E)$ is log canonical. Hence $A_{X,\Delta}(F)\geq\ord_F(\m_x)$ for every prime divisor $F$ centered at $x$, so $b_x(X,\Delta)\geq1$. On the other hand, $A_{X,\Delta}(E)=\ord_E(\m_x)=1$, and therefore $b_x(X,\Delta)\leq1$. Thus $b_x(X,\Delta)=1$.

The germ $x\in W$ has multiplicity $\deg S=4$ and embedding dimension $5$. Thus $n=5$, $d=3$, $c=2$, $b=1$, $a=1$, and $s=1$. Corollary~\ref{cor:discrete-lc-center} is sharp in this example: $\mult_xW=4$. Theorem~\ref{thm:multiplicity-bound} bounds $\mult_xW$ by $27/4$, while Helmke's multiplicity bound in \cite[Theorem~4.3]{Hel97} is $\mult_xW\leq\binom{4}{2}=6$.
\end{remark}

The example in Remark~\ref{rem:discrete-lc-center} shows that the bound in Corollary~\ref{cor:discrete-lc-center} can be attained. It is therefore natural to ask when the bound is sharp.

We call a triple $(d,c,b)$ \emph{smoothly realizable} if there exist a smooth germ $x\in X$, a pair $(X,\Delta)$ that is log canonical but not klt at $x$, and a positive-dimensional minimal lc center $W=W_x(X,\Delta)$ such that $\dim W=d$, $\edim\mathcal O_{W,x}-d=c$, and $b_x(X,\Delta)=b$. For a smoothly realizable triple, let $M_{d,c}(b)$ denote the maximal possible value of $\mult_xW$.

\begin{question}
Characterize the smoothly realizable triples $(d,c,b)$ and determine $M_{d,c}(b)$. In particular, when is the bound in Corollary~\ref{cor:discrete-lc-center} attained?
\end{question}

A sharper multiplicity bound for the minimal lc centers arising in Proposition~\ref{prop:cutting}
may improve the numerical estimate and reduce $C_0$ in Theorem~\ref{thm:main}.

Finally, we prove Corollary~\ref{cor:klt-multiplicity}.

\begin{proof}[Proof of Corollary~\ref{cor:klt-multiplicity}]
After shrinking $X$ around $x$, we may assume that
$(X,\Delta)$ is klt.
Let $p\colon Y:=X\times\mathbb A^1\to X$ be the projection,
and set $H:=X\times\{0\}$ and $y:=(x,0)$.

Take a log resolution $\mu\colon X'\to X$ of
$(X,\Delta,\m_x)$ and write
$K_{X'}+\Delta_{X'}=\mu^*(K_X+\Delta)$.
Let $p'\colon X'\times\mathbb A^1\to X'$ be the projection.
Then
$$
K_{X'\times\mathbb A^1}+p'^*\Delta_{X'}+X'\times\{0\}
=(\mu\times\operatorname{id}_{\mathbb A^1})^*
(K_Y+p^*\Delta+H).
$$
Since all coefficients of $\Delta_{X'}$ are less than $1$
and $p'^*\Delta_{X'}+X'\times\{0\}$ has simple normal
crossing support, $(Y,p^*\Delta+H)$ is log canonical
and has $H$ as its unique lc center.

Write $\m_x\mathcal O_{X'}=
\mathcal O_{X'}(-\sum_jr_jE_j)$, where $r_j>0$. On the log resolution
$\mu\times\operatorname{id}_{\mathbb A^1}$, the pullback
of $\m_x\mathcal O_Y$ has order $r_j$ along $E_j\times\mathbb A^1$ and order $0$ along $X'\times\{0\}$.
Hence
$$
\lct_y((Y,p^*\Delta+H);\m_x\mathcal O_Y)
=\min_j\frac{1-\mult_{E_j}\Delta_{X'}}{r_j}=b.
$$
Set $b':=b_y(Y,p^*\Delta+H)$. Since $\m_x\mathcal O_Y\subseteq\m_y$, we have $b'\geq b$. The local rings $\mathcal O_{H,y}$ and $\mathcal O_{X,x}$
are isomorphic and are Cohen--Macaulay by
Lemma~\ref{lem:minimal-lc-center}.
By Theorem~\ref{thm:local} and
Lemma~\ref{lem:ordinary-normal-e1}, we have
$$
e_1(\m_x)\leq\ebarone(\m_x)
\leq\frac{d-b'}{2}\mult_xX
\leq a\mult_xX.
$$
Proposition~\ref{prop:discrete-artinian} now proves
the corollary.
\end{proof}

\begin{remark}\label{rem:klt-multiplicity-without-edim}
Similarly, we may show that the inequality in
Corollary~\ref{cor:multiplicity-without-edim} also holds for klt pairs.
More precisely, let $(X,\Delta)$ be a pair that is
klt near a closed point $x$, and set $d:=\dim X$. Then
$$b_x(X,\Delta)\leq d-2+\frac{2}{\mult_xX}.$$
The following example shows that equality can hold in the above inequality. Let $d,r\geq2$ be integers, let
$S:=\mathbb A^2/\boldsymbol\mu_r$, where
$\zeta\cdot(u,v)=(\zeta u,\zeta v)$, and let
$x\in X:=S\times\mathbb A^{d-2}$ be the image of
the origin.
The Veronese grading gives $\mult_xX=r$.
The quotient map $\mathbb A^d\to X$ is \'etale in codimension one, and the pullback of $\m_x$ is
$(u,v)^r+(z_1,\ldots,z_{d-2})$.
The log canonical threshold of this monomial ideal is
$d-2+2/r$, so
$b_x(X,0)=d-2+2/r$
(cf. \cite[Proposition~5.20]{KM98}).
Thus equality holds in the above inequality. Moreover, $c=r-1$, $a=1-1/r$, and $s=1$, so the bound
in Corollary~\ref{cor:klt-multiplicity} is also attained.
\end{remark}



\end{document}